\documentclass[10pt,reqno]{amsproc}
\newtheorem{theorem}{\bf{Theorem}}[section] 
\newtheorem{lemma}[theorem]{\bf{Lemma}}     
\newtheorem{ex}[theorem]{\bf{Example}}
\newtheorem{corollary}[theorem]{\bf{Corollary}}
\newtheorem{proposition}[theorem]{\bf{Proposition}}
\newtheorem{definition}[theorem]{\bf{Definition}}

\usepackage{color}

\title[ON THE EXTENDABILITY OF SOME CLASSES OF MAPS ON HILBERT $C^*$-MODULES]
 {ON THE EXTENDABILITY OF SOME CLASSES OF MAPS ON HILBERT $C^*$-MODULES}  

\subjclass[2010]{Primary: 46L08, Secondary: 46L07.}
\keywords{Hilbert $C^*$-modules, Wittsock's extension theorem, completely
positive maps, completely semi-$\phi$-maps}

 \author[M. B. Asadi]{Mohammad B. Asadi}
\address{School of Mathematics, Statistics and Computer Science,
 College of Science, University of Tehran, Tehran,
  Iran, and \\
   School of Mathematics, Institute for Research in Fundamental Sciences (IPM),
P.O. Box: 19395-5746, Tehran, Iran}

\email{mb.asadi@khayam.ut.ac.ir}

\author[R. Behmani]{Reza Behmani}
\address{\noindent Department of Mathematics, Kharazmi University, 50, Taleghani Ave.,15618, Tehran Iran}

\email{reza.behmani@gmail.com}

\author[A. R. Medghalchi]{Ali R. Medghalchi}
\address{\noindent Department of Mathematics, Kharazmi University, 50, Taleghani Ave.,15618, Tehran IRAN.}

\email{a\_medghalchi@khu.ac.ir}%

\author[H. Nikpey]{Hamed Nikpey}
\address{Department of Mathematics, Shahid Rajaei Teacher Training University, Tehran 16785-136, Iran}
\email{hamednikpey@gmail.com}

\begin{document}
\maketitle

\begin{abstract}
In this paper, we show that every completely semi-$\phi$-map on a
submodule of a Hilbert $C^*$-module has a completely
semi-$\phi$-map extension on the whole of module. We also
investigate the extendability of $\phi$-maps and provide examples
of $\phi$-maps which has no $\phi$-map extension. Finally, we
introduce a category of Hilbert $C^*$-module and determine
injective objects in this category.
\end{abstract}



\section{INTRODUCTION}

One of the most fundamental theorems in the theory of operator
spaces is the Wittsock's extension theorem for completely bounded
maps which is the noncommutative counterpart of the celebrated
Hann-Banach's Extension theorem.  The authors in \cite{ABN},
introduced the concept of completely semi-$\phi$-maps as a
generalization of $\phi$-maps. Also, they shown that every operator
valued completely bounded linear map on a Hilbert $C^*$-module is a
completely semi-$\phi$-map for some completely positive map $\phi$
on the underlying $C^*$-algebra of the Hilbert $C^*$-module and vice
versa \cite{BM}. Thus it is natural to seeking for an analogue of
Wittsock's extension theorem for completely semi-$\phi$-maps.

In this note, we show that every $\phi$-map or completely
semi-$\phi$-map on a submodule of a Hilbert $C^*$-module has a
completely semi-$\phi$-map extension on the whole of the Hilbert
$C^*$-module.
 Furthermore, we provide examples of some $\phi$-map
which has no $\phi$-map extension on the whole of module.
 However, for some special case of
$\phi$-maps, we will show to how construct a completely
semi-$\phi$-map extension of a $\phi$-map which is close to being a
$\phi$-map.

In the last section, we introduce a category of Hilbert
$C^*$-module and determine injective objects in this category.

For every Hilbert spaces $\mathcal{H,K},$ the set of all bounded
operators $\mathcal{B(H,K)}$ is a right Hilbert
$\mathcal{B(H)}$-module, where the module action is the composition
of operators and the $\mathcal{B(H)}$-inner product on
$\mathcal{B(H,K)}$ is given by $\langle T,S\rangle=T^*S$ for every
$S,T\in\mathcal{B(H,K)}.$

Assume $\mathcal{A}$ and $\mathcal{B}$ are  $C^*$-algebras,
$\mathcal{E}$ and $\mathcal{G}$ are right Hilbert $C^*$-modules over
$\mathcal{A}$ and $\mathcal{B}$ respectively,
$\phi:\mathcal{A}\rightarrow\mathcal{B}$ is a completely positive
map and $\Phi:\mathcal{E}\rightarrow\mathcal{G}$ is a map, we say

(1) $\Phi$ is a \textit{$\phi$-map}, if
$\langle\Phi(x),\Phi(y)\rangle= \phi(\langle x,y\rangle)$, for all
$x,y\in\mathcal{E}$.

(2) $\Phi$ is a \textit{$\phi$-morphism}, if $\Phi$ is a $\phi$-map
and $\phi$ is a $*$-homomorphism.

(3) $\Phi$ is a \textit{completely semi-$\phi$-map}, if
$\langle\Phi_n(x),\Phi_n(x)\rangle\leq\phi_n(\langle x,x\rangle)$
 for every $x\in\mathbb{M}_n(\mathcal{E})$ and $n\in\mathbb{N}.$

When $\mathcal{G}=\mathcal{B(H,K)}$ and $\mathcal
A=\mathcal{B(H)}$ for some Hilbert spaces $\mathcal{H,K},$ a
$\phi$-morphism is called \textit{$\phi$-representation} and in
this case

(4) $\Phi$ is \textit{non-degenerate}, if
$[\Phi(\mathcal{E})\mathcal{H}]=\mathcal{K}.$

Note that a $\phi$-morphism  $\Phi$ is linear and satisfies
$\Phi(x.a)=\Phi(x)\phi(a)$ for every $x\in\mathcal{E}$ and
$a\in\mathcal{A},$ therefore $\Phi$ is a ternary morphism (triple
morphism), that is $\Phi(x\langle
y,z\rangle)=\Phi(x)\langle\Phi(y),\Phi(z)\rangle$ for all
$x,y,z\in\mathcal{E}.$
 For more information on representation theory
of Hilbert $C^*$-modules,  $\phi$-maps and their dilation theory
refer to \cite{Arambašic}, \cite{Asadi}, \cite{BM}, \cite{BRS}, \cite{skeide} and
\cite{Skeide-Sumesh}.

\section{EXTENDABILITY OF COMPLETELY SEMI-$\phi$-MAPS}
Throughout this section, we assume that $\mathcal{H},\mathcal{H}_1,
\mathcal{H}_2$ are Hilbert spaces, $\mathcal{A}$ is a $C^*$-algebra
and $\mathcal{F}$ is a non-trivial closed submodule of a Hilbert
$\mathcal{A}$-module $\mathcal{E}$.
 Also we note that the set
$\mathcal{F}^{\perp}= \{x \in \mathcal{E} ~|~ \langle x, y\rangle
=0,~ \text{for all} ~ y \in \mathcal{F} \}$ is a  closed submodule
of  $\mathcal{E}$.

In this section, we concentrate on operator valued maps on Hilbert
$C^*$-modules. In fact, if
$\phi:\mathcal{A}\rightarrow\mathcal{B}(\mathcal{H}_{1})$ is a
completely positive map, then an \textit{operator valued $\phi$-map
on $\mathcal{F}$} as $\Phi$, means that $\Phi$ is a $\phi$-map from
$\mathcal{F}$ into $\mathcal{B}(\mathcal{H}_1)$-module
$\mathcal{B}(\mathcal{H}_1,\mathcal{H}_2)$.

We first discuss the extension problem for $\phi$-maps.
 In the following,
  we provide an example of some $\phi$-map on a submodule of a
Hilbert $C^*$-module which can not be extended to any $\phi$-map on
the whole of the  Hilbert $C^*$-module.

\begin{ex}\label {e1}
Suppose that $\mathcal{K}(\mathcal{H})$ is the set of all compact
operators on $\mathcal{H}$. Clearly, $\mathcal{E}=
\mathcal{K}(\mathcal{H}) \oplus \mathcal{K}(\mathcal{H})$ is a full
Hilbert $\mathcal{K}(\mathcal{H})$-module (by
$\mathcal{K}(\mathcal{H})$-valued inner product $\langle (T_1, S_1),
(T_2, S_2) \rangle = T_{1}^* T_2 + S_{1}^* S_2$)
 and $\mathcal{F}=
\mathcal{K}(\mathcal{H}) \oplus 0 $ is a nontrivial Hilbert
submodule of $\mathcal{E}$. Consider the inclusion map
$\phi=id:K(\mathcal{H})\rightarrow
K(\mathcal{H})\subset\mathcal{B(H)}.$ The map $\Phi: \mathcal{F} \to
\mathcal{B}(\mathcal{H}, \mathcal{H})$ defined by $\Phi((T, 0))=T$
is a $\phi$-map which doesn't  have any $\phi$-map extension on
$\mathcal{E}$.

 In fact, if $\Phi': \mathcal{E} \to
\mathcal{B}(\mathcal{H}, \mathcal{H})$ is a $\phi$-map extension of
$\Phi$,  then
 $$\Phi'( (T_1, S_1))^* \Phi'(T_2, S_2) = T_{1}^* T_2 +
S_{1}^* S_2 ~ \text{and} ~ \Phi'( (T_1, 0))= T_1,$$ for all $T_{1},
T_2, S_{1}, S_2 \in \mathcal{K}(\mathcal{H})$. A directly
calculation shows that $\Phi'( (0, S))=0$ for all $S \in
\mathcal{K}(\mathcal{H})$. Consequently, $\Phi' = \Phi \oplus 0$ and
so $\Phi'$ is not a $\phi$-map on $\mathcal{E}$
\end{ex}

The following lemma provides a necessary condition for a
completely positive map $\phi$, such that every $\phi$-map on a
submodule has a $\phi$-map extension on the whole of the Hilbert
$C^*$-module.

\begin{lemma}\label{l1}
If $\phi:\mathcal{A}\rightarrow\mathcal{B}(\mathcal{H}_{1})$ is a
completely positive map and there exists a non-degenerate $\phi$-map
$\Phi:\mathcal{F}\to\mathcal{B}(\mathcal{H}_1,\mathcal{H}_2)$  which
has a $\phi$-map extension on $\mathcal{E}$, then

(i) $\phi(\langle \mathcal{F}^{\perp},\mathcal{E}\rangle)=\{0\},$

(ii) every  operator valued $\phi$-map on $\mathcal{F}$ has a
$\phi$-map extension on $\mathcal{E}.$
\end{lemma}
\begin{proof}
(i)  Assume there is a $\phi$-map
$\Psi:\mathcal{E}\to\mathcal{B}(\mathcal{H}_1,\mathcal{H}_2),$
extending $\Phi.$ For every $h,h'\in\mathcal{H}_1$ and
$x\in\mathcal{F},\hspace{2mm}z\in\mathcal{F}^{\bot}$ one has
$$\langle\Psi(z)h,\Phi(x)h'\rangle=\langle\Psi(z)h,\Psi(x)h'\rangle=
\langle\Psi(x)^*\Psi(z)h,h'\rangle=\langle\phi(\langle
x,z\rangle)h,h'\rangle=0.$$ By the assumption, $\Phi$ is a
non-degenerate map and therefore $\Psi(z)h=0$, thus
$\Psi(\mathcal{F}^{\bot})=\{0\}.$ On the other hand, $\Psi$ is a
$\phi$-map, so for every $x\in\mathcal{E},y\in\mathcal{F}^{\bot}$
$$\phi(\langle y, x\rangle)=\Psi(y)^*\Psi(x)=0.$$
Then $\phi(\langle \mathcal{F}^{\bot},\mathcal{E}\rangle)= \{0\}.$

(ii) Assume
$\Theta:\mathcal{F}\to\mathcal{B}(\mathcal{H}_1,\mathcal{K})$ is a
 $\phi$-map. By \cite[Theorem 2.2]{BM}, there is an isometry
$S:\mathcal{H}_2\to\mathcal{K}$ such that $S\Phi(x)=\Theta(x)$,
for every $x\in\mathcal{F}.$ Define
$\Theta':\mathcal{E}\to\mathcal{B}(\mathcal{H}_1,\mathcal{K})$ by
$\Theta'(x):=S\Psi(x)$ for each $ x\in\mathcal{E}.$ Since $\Psi$
is a $\phi$-map extension of $\Phi$ and $S$ is an isometry,
$\Theta'$ is a $\phi$-map extension of $\Theta.$

\end{proof}

By Kolmogorov's decomposition theorem, for every completely positive
map $\phi:\mathcal{A}\rightarrow\mathcal{B}(\mathcal{H}_{1})$, there
is at least a non-degenerate operator valued $\phi$-map on
$\mathcal{F}$. Therefore,

\begin{corollary}\label{c1}
Assume $\phi:\mathcal{A}\rightarrow\mathcal{B}(\mathcal{H}_{1})$ is
a completely positive map. If every operator valued $\phi$-map on
$\mathcal{F}$ has a $\phi$-map extension on $\mathcal{E},$ then
$\phi(\langle \mathcal{F}^{\perp},\mathcal{E}\rangle)=\{0\}.$
\end{corollary}

 \begin{theorem}\label{t2}
Suppose that
$\phi:\mathcal{A}\rightarrow\mathcal{B}(\mathcal{H}_{1})$ is a
completely positive map and
$\Phi:\mathcal{F}\rightarrow\mathcal{B}(\mathcal{H}_{1},\mathcal{H}_{2})$
is a completely semi-$\phi$-map. Then
  $\Phi$   has a  completely semi-$\phi$-map extension
 $\Phi':\mathcal{E}\rightarrow\mathcal{B}(\mathcal{H}_{1},\mathcal{H}_{2}).$
 Furthermore, if $\Phi$ is a $\phi$-map and  $\phi(\langle \mathcal{F}^{\bot}, \mathcal{E}
\rangle)=0$, then

(i) $\Phi'(\mathcal{F}^{\bot})=\{0\},$

(ii) $\Phi'(x)^*\Phi'(y)=\phi(\langle x,y\rangle)$ and $
\Phi'(y)^*\Phi'(x)=\phi(\langle y,x\rangle),$ for all
$x\in\mathcal{E}, y\in\mathcal{F}\oplus\mathcal{F}^{\bot}$.
  \end{theorem}
 \begin{proof}
 Let $(\rho,\mathcal{K}_{1},V)$ be a minimal Stinespring's dilation triple for $\phi.$
There is a triple
$((\Phi_{\phi},\mathcal{H}_\phi),(\Psi_{\rho},\mathcal{K}_\rho),W_\phi)$
consists of a non-degenerate $\phi$-map
$\Phi_{\phi}:\mathcal{E}\to\mathcal{B}(\mathcal{H}_1,\mathcal{H}_\phi),$
a non-degenerate $\rho$-representation
$\Psi_{\rho}:\mathcal{E}\to\mathcal{B}(\mathcal{K}_1,\mathcal{K}_\rho)$
and a unitary operator
$W_{\phi}:\mathcal{H}_{\phi}\to\mathcal{K}_{\rho}$ such that
satisfies $W_\phi\Phi_\phi(x)=\Psi_\rho(x)V$ for every
$x\in\mathcal{E}$ by \cite[Theorem 2.2 part(i)]{BM}.
 Since $\Phi$ is a completely semi-$\phi$-map, we have
$[\Phi(x_i)^*\Phi(x_j)]_{i,j}\leq [\phi(\langle
x_i,x_j\rangle)]_{i,j},$ for every $x_1,...,x_n\in\mathcal{F}.$
Therefore, for every
 $h_1,...,h_n\in\mathcal{H}_1$ we have
 $$\|\sum_{i=1}^n\Phi(x_i)h_i\|^2\leq\sum_{i=1}^n\sum_{j=1}^n\langle\phi(x_j,x_i)h_i,h_j\rangle=
 \|\sum_{i=1}^n\Phi_{\phi}(x_i)h_i\|^2.$$
Thus there is a unique contractive linear operator
$S_0:[\Phi_\phi(\mathcal{F})\mathcal{H}_1]\to\mathcal{H}_2$ such
that $S_0\Phi_\phi(x)h=\Phi(x)h$ for every $x\in\mathcal{F}$ and
$h\in\mathcal{H}_1.$ Let $P\in\mathcal{B}(\mathcal{H}_\phi)$ be the
orthogonal projection onto $[\Phi_\phi(\mathcal{F})\mathcal{H}_1].$
Put $S:=S_0P:\mathcal{H}_\phi\to\mathcal{H}_2.$
  Therefore $\Phi(x)=SW_\phi^*\Psi_\rho(x)V$ for every $x\in\mathcal{F}.$
 Put $W:=W_{\phi}S^*$ and define
 $\Phi':\mathcal{E}\rightarrow\mathcal{B}(\mathcal{H}_{1},\mathcal{H}_{2})$
 by
 $$\Phi'(x):=W^*\Psi_{\rho}(x)V$$
 for all $x\in\mathcal{E}.$
  Since $W$ is a contraction, $\Phi'$ is a completely semi-$\phi$-map by \cite[Theorem 2.2 part (iii)]{BM}.
  Obviously,  $\Phi'$ is an extension for $\Phi.$

Now, let $\Phi$ be a $\phi$-map and  $\phi(\langle
\mathcal{F}^{\bot}, \mathcal{E} \rangle)=0$. In this case, the above
inequality becomes equality and so $S_0$ becomes an isometry.
 Also,
$0 \leq\Phi'(x)^*\Phi'(x)\leq\phi(\langle x,x\rangle)=0$ satisfies
for every $x\in\mathcal{F}^\perp.$
   Therefore $\Phi'(\mathcal{F}^\perp)=\{0\}.$

Finally it must be shown that $\Phi'$ satisfies
$\Phi'(x)^*\Phi'(y)=\phi(\langle x,y\rangle)$ for all
$x\in\mathcal{E}\hspace{1mm},\hspace{1mm}y\in\mathcal{F}\oplus\mathcal{F}^{\bot}.$
 For this, assume  $x\in\mathcal{E},\hspace{1mm}y\in\mathcal{F}$ and $h\in\mathcal{H}_{1},$
  we have
 $$WW^*\Psi_\rho(y)Vh=W_\phi S^*SW^*_\phi\Psi_\rho(y)Vh=W_\phi P\Phi_\phi(y)h=W_\phi\Phi_\phi(y)h=\Psi_\rho(y)Vh,$$
   therefore
   $$\Phi'(x)^*\Phi'(y)h=V^*\Psi_\rho(x)^*WW^*\Psi_\rho(y)Vh=
   V^*\Psi_\rho(x)^*\Psi_\rho(y)Vh=V^*\rho(\langle x,y\rangle)Vh=
   \phi(\langle x,y\rangle)h.$$
Thus $\Phi'(x)^*\Phi'(y)=\phi(\langle x,y\rangle)$ for every
$x\in\mathcal{E}$ and $y\in\mathcal{F}.$ Since
$\Phi'(\mathcal{F}^{\bot})=\{0\}$ and
$\phi(\langle\mathcal{E},\mathcal{F}^{\bot}\rangle)=\{0\},$ for
every $x\in\mathcal{E},\hspace{1mm}y\in\mathcal{F},\hspace{1mm}
z\in\mathcal{F}^{\bot}$ we have
 $\Phi'(x)^*\Phi'(z)=0=\phi(\langle x,z\rangle),$ and
  therefore
 $$\Phi'(x)^*\Phi'(y+z)=\Phi'(x)^*\Phi'(y)+\Phi'(x)^*\Phi'(z)=
 \phi(\langle x,y\rangle)+\phi(\langle x,y\rangle)=\phi(\langle x,y+z\rangle).$$

\end{proof}

 The following corollary says that if a non-degenerate operator valued $\phi$-map
 has a $\phi$-map extension, then the extension is
unique.

\begin{corollary}\label{c2}
Let $\phi:\mathcal{A}\to\mathcal{B}(\mathcal{H}_{1})$ be
   a completely positive map,
$\Phi:\mathcal{F}\to\mathcal{B}(\mathcal{H}_{1},\mathcal{H}_{2})$ a
non-degenerate $\phi$-map and
$\Gamma:\mathcal{E}\to\mathcal{B}(\mathcal{H}_{1},\mathcal{H}_{2})$
be a $\phi$-map such that $\Gamma|_{\mathcal{F}}=\Phi.$
 Then $\Gamma=\Phi'$, where $\Phi'$ is as in the proof of Theorem \ref{t2}.
\end{corollary}
\begin{proof}
We use the notions of the proof of Theorem \ref{t2}. Since $\Gamma$
is an $\phi$-extension of $\Phi,$ and $\Phi$ is  a non-degenerate
map, then $\Gamma$ is non-degenerate and
$\mathcal{H}_{2}=[\Phi(\mathcal{F})\mathcal{H}_{1}]=[\Gamma(\mathcal{E})\mathcal{H}_{1}].$
 By \cite[Theorem 2.2]{BM}, there is  a unitary operator
   $W':\mathcal{H}_{2}\to\mathcal{K}_{\rho}$ such that
$W'\Gamma(e)h:=\Psi_\rho(e)Vh$ for all
$e\in\mathcal{E},\hspace{1mm}h\in\mathcal{H}_{1}.$
 Thus
$$W'\Phi(f)h=W'\Gamma(f)h=\Psi_{\rho}(f)Vh=W\Phi'(f)h=W\Phi(f)h$$
for all $f\in\mathcal{F},\hspace{1mm}h\in\mathcal{H}_{1}.$
 Therefore $W'=W$.

  Now, if $x,y\in\mathcal{E}$ and $h\in\mathcal{H}_{1},$ then we have
\begin{equation*}\begin{split}
\Gamma(x)^*\Gamma(y)h &=\phi(\langle x,y\rangle)h=V^*\rho(\langle
x,y\rangle)V h=
V^*\Psi_\rho(x)^*\Psi_\rho(y)Vh\\&=V^*\Psi_\rho(x)^*W'\Gamma(y)h=V^*\Psi_\rho(x)^*W\Gamma(y)h=\Phi'(x)^*\Gamma(y)h.
\end{split}\end{equation*}
Since $\Gamma(x)^*$ and $\Phi'(x)^*$ are bounded operators and
$[\Gamma(\mathcal{E})\mathcal{H}_{1}]=\mathcal{H}_{2},$ we have
$\Gamma(x)^*=\Phi'(x)^*.$
\end{proof}

If $\mathcal{A}$  is a $C^*$-subalgebra of $K(\mathcal{H})$, then it
is well known that, every closed submodule $\mathcal{F}$ of
$\mathcal{E}$ satisfies the equations
$\mathcal{F}^{\perp\perp}=\mathcal{F}$ and
$\mathcal{F}\oplus\mathcal{F}^{\perp}=\mathcal{E}.$ Therefore, by
Theorem \ref{t2} and Corollary \ref{c2} we can conclude  that the
necessary condition
$\phi(\langle\mathcal{F}^{\perp},\mathcal{E}\rangle)=0$ in Lemma
\ref{l1}, is a sufficient condition for existence of $\phi$-map
extension, in this case.

\begin{corollary}\label{c3}
If $\mathcal{A}$  is a $C^*$-algebra of compact operators
 and
$\phi:\mathcal{A}\to\mathcal{B}(\mathcal{H}_{1})$  is a completely
positive map and $\mathcal{E}$ is a full Hilbert
$\mathcal{A}$-module.
  Then the following statements are equivalent:

 (i) $\phi(\langle\mathcal{F}^{\perp},\mathcal{E}\rangle)=0,$

 (ii) there is a non-degenerate
  operator valued $\phi$-map
on $\mathcal{F}$ which has a $\phi$-map extension on $\mathcal{E}$,

 (iii) every operator valued $\phi$-map
on $\mathcal{F}$
   has a $\phi$-map extension on
$\mathcal{E}.$

(iv) for every $\phi$-map
$\Phi:\mathcal{F}\to\mathcal{B}(\mathcal{H}_{1},\mathcal{H}_{2}),$
the map $\Phi'=\Phi \oplus
0:\mathcal{E}=\mathcal{F}\oplus\mathcal{F}^{\perp}\to\mathcal{B}(\mathcal{H}_{1},\mathcal{H}_{2})$
is a $\phi$-map and also $\Phi'$ is the unique $\phi$-map extension
of $\Phi$ on $\mathcal{E}.$

\end{corollary}

Since the $C^*$-algebra $\mathcal{K}(\mathcal{H})$ is simple, every
nonzero Hilbert $\mathcal{K}(\mathcal{H})$-module is full. In
particular, $\langle\mathcal{F}^{\perp},\mathcal{E}\rangle =
\langle\mathcal{F}^{\perp},\mathcal{F}^{\perp}\rangle =
\mathcal{K}(\mathcal{H})$. Therefore,
$\phi(\langle\mathcal{F}^{\perp},\mathcal{E}\rangle)\neq0,$ for
every nonzero completely positive map
$\phi:\mathcal{K}(\mathcal{H})\to\mathcal{B}(\mathcal{H}_{1})$.
Hence we have
\begin{corollary}\label{c4}
If $\mathcal{A}=\mathcal{K}(\mathcal{H})$ and
$\phi:\mathcal{A}\to\mathcal{B}(\mathcal{H}_{1})$  is a nonzero
completely positive map, then any operator valued $\phi$-map on
$\mathcal{F}$
 has no
$\phi$-map extension on $\mathcal{E}.$
\end{corollary}


\section{Category of Hilbert $C^*$-modules and Completely semi-$\phi$-maps}

In the following, we define the category $\mathcal{C}_{H,C^*}$ as
a category whose objects are pairs $(\mathcal{E},\mathcal{A})$
where $\mathcal{A}$ is a $C^*$-algebra and $\mathcal{E}$ is  a
right Hilbert $\mathcal{A}$-module and a morphism from
$(\mathcal{E}_1,\mathcal{A}_1)$ to $(\mathcal{E}_2,\mathcal{A}_2)$
is a pair $(\Phi,\phi)$ consists of a completely positive map
$\phi:\mathcal{A}\rightarrow\mathcal{B}$ and a completely
semi-$\phi$-map $\Phi:\mathcal{E}_1\rightarrow\mathcal{E}_2$ and
the composition of two morphisms $(\Phi,\phi)$ and $(\Psi,\psi)$
is $(\Phi,\phi)\circ(\Psi,\psi):=(\Phi\circ\Psi,\phi\circ\psi).$
 If we restrict ourselves to the case of
full Hilbert $C^*$-modules over unital $C^*$-algebras and unital
completely positive maps, we obtain a subcategory of
$\mathcal{C}_{H,C^*}$ which we denote it by
$\mathcal{C}^{1}_{H,C^*}.$ In the following we generalize some
results on the characterization of completely semi-$\varphi$-maps
and use it to better understanding $\mathcal{C}^{1}_{H,C^*}$ as  a
subcategory of operator systems $\mathcal{C}_\mathcal{OS},$ and
characterize its injective objects. Finally, we compare this new
category with the category of Hilbert $C^*$-modules when its
morphisms are $\phi$-maps, completely bounded maps or Hilbert
modules morphisms.

For a Hilbert $C^*$-module $\mathcal{E}$ over a $C^*$-algebra
$\mathcal{A}$, the smallest operator system which contains
$\mathcal{A}$ and $\mathcal{E}$ is denoted by
$S_{\mathcal{A}}(\mathcal{E})$ and is defined as follow
$S_{\mathcal{A}}(\mathcal{E}):=\begin{bmatrix}
\mathbb{C}I_{\mathcal{E}} & \mathcal{E} \\
\mathcal{E}^* & \mathcal{A}
\end{bmatrix}=\{\begin{bmatrix}
\lambda & x \\
y^* & a
\end{bmatrix} | \ a\in \mathcal{A}, \ \lambda\in\mathbb{C}, \ x,y\in\mathcal{E}\}.$
The following theorem is a generalization of \cite[Lemma 3.2]{ABN}
which is useful in the study of $\mathcal{C}_{H,C^*}.$


\begin{proposition}\label{p3.1}
Suppose that $\mathcal{E}$ and $\mathcal{F}$ are right Hilbert
$C^*$-modules over the $C^*$-algebras $\mathcal{A,B}$,
respectively, and also $\phi:\mathcal{A}\rightarrow\mathcal{B}$ is
a completely positive map and
$\Phi:\mathcal{E}\rightarrow\mathcal{F}$ is a linear map. Then,
$\Phi$ is a completely semi-$\phi$-map if and only if
$$\begin{bmatrix}
id & \Phi \\
\Phi^* & \phi
\end{bmatrix}:S_\mathcal{A}(\mathcal{E})\rightarrow S_\mathcal{B}(\mathcal{F})
 \ \ (\text{given by} \ \ \begin{bmatrix}
\lambda & x \\
y^* & a
\end{bmatrix}\mapsto \begin{bmatrix}
\lambda  & \Phi(x) \\
\Phi(y)^* & \phi(a)
\end{bmatrix})$$
 is a completely positive map.
\end{proposition}
\begin{proof}
The same argument as in the proof of \cite[Lemma 3.2]{ABN}, works
here.
\end{proof}

Hence we have the following result.

 \begin{theorem}\label{t3.2}
 $\mathcal{C}^{1}_{H,C^*}$ is (up to isomorphism) a subcategory of $\mathcal{C}_\mathcal{OS},$ the category of operator systems.
\end{theorem}
\begin{proof}
Define the map
$\Sigma:\mathcal{C}^{1}_{H,C^*}\rightarrow\mathcal{C}_\mathcal{OS}$
which corresponds to  every object $(\mathcal{E,A})$ of
$\mathcal{C}^{1}_{H,C^*},$ the operator system
$S_\mathcal{A}(\mathcal{E}),$ and corresponds to every morphism
$(\Phi,\phi)$ between two objects of $\mathcal{C}^{1}_{H,C^*}$
such as $(\mathcal{E,A})$ and $(\mathcal{F,B}),$ the unital
completely positive map $\begin{bmatrix}
id & \Phi \\
\Phi^* & \phi
\end{bmatrix}:S_\mathcal{A}(\mathcal{E})\rightarrow S_\mathcal{B}(\mathcal{F})$.
It is easy to check that for
$(\Phi_1,\phi_1):(\mathcal{E}_1,\mathcal{A}_1)\rightarrow(\mathcal{E}_2,\mathcal{A}_2)$
and
$(\Phi_2,\phi_2):(\mathcal{E}_2,\mathcal{A}_2)\rightarrow(\mathcal{E}_3,\mathcal{A}_3)$
$\Sigma((\Phi_2,\phi_2)\circ(\Phi_1,\phi_1))=\Sigma((\Phi_2,\phi_2))\circ\Sigma((\Phi_1,\phi_1)).$
Thus $\Sigma$ is a one-to-one covariant functor.
\end{proof}

Therefore, we can consider $\mathcal{C}_{H,C^*}^1$ as a category
consists of block-wise operator systems $\begin{bmatrix}
\mathbb{C}I_{\mathcal E} & \mathcal E\\
\mathcal{E}^* & \mathcal A \end{bmatrix},$ where $\mathcal A$ is a
unital $C^*$-algebra and $\mathcal E$ is a full right Hilbert
$\mathcal A$-module, and morphisms are corner preserving unital
completely positive maps.

We remark that there is some completely positve map between
operator systems $S_{\mathcal{A}}(\mathcal{E})$ and
$S_{\mathcal{B}}(\mathcal{F})$  which is not corner preserving.


\begin{ex}\label{e3.4}
For a given Hilbert space $\mathcal H$ and every bounded operators
$T_1,T_2,T_3,T_4$ on it, by elementary row and column operations
we have the following unitary equivalence in
$\mathcal{B}(\mathcal{H}^4)$
$$\begin{bmatrix}
T_1 & 0 & 0 & T_2 \\
0 & T_4 & 0 & 0 \\
0 & 0 & T_1 & 0 \\
T_3 & 0 & 0 & T_4
\end{bmatrix}\cong\begin{bmatrix}
T_1 & T_2 & 0 & 0 \\
T_3 & T_4 & 0 & 0 \\
0 & 0 & T_1 & 0 \\
0 & 0 & 0 & T_4
\end{bmatrix}. $$
Therefore the map $\varphi:\mathcal{B}(\mathcal{H}^2)\rightarrow\mathcal{B}(\mathcal{H}^4)$ defined by
$$\varphi(\begin{bmatrix}
T_1 & T_2\\
T_3 & T_4
\end{bmatrix}):=\begin{bmatrix}
T_1 & 0 & 0 & T_2 \\
0 & T_4 & 0 & 0 \\
0 & 0 & T_1 & 0 \\
T_3 & 0 & 0 & T_4
\end{bmatrix}$$
is a unital completely positive map which is not
corner-preserving. Now considering $\mathcal{B}(\mathcal{H}^i)$ as
Hilbert $C^*$-module over itself, for $i=2,4,$ and restriction of
$\varphi$ on
$S_{\mathcal{B}(\mathcal{H}^2)}(\mathcal{B}(\mathcal{H}^2))$
provides an example of a unital completely positive map
$\varphi:S_{\mathcal{B}(\mathcal{H}^2)}(\mathcal{B}(\mathcal{H}^2))
\rightarrow
S_{\mathcal{B}(\mathcal{H}^4)}(\mathcal{B}(\mathcal{H}^4))$ which
is not corner preserving, thus it is not a morphism in
$\mathcal{C}_{H,C^*}^1.$

\end{ex}


\begin{definition}\label{d3.5}
Let $(\mathcal{E,A})$ and $(\mathcal{F,B})$ be two objects of
$\mathcal{C}_{H,C^*}.$ We say that $(\mathcal{E,A})$ contained in
$(\mathcal{F,B})$ (or $(\mathcal{F,B})$ contains
$(\mathcal{E,A})),$ and denote it by
$(\mathcal{E,A})\subset(\mathcal{F,B}),$ when $\mathcal{A}$ is a
$C^*$-subalgebra of $\mathcal{B}$ and
 $\mathcal{E}\subseteq\mathcal{F}$ and $\langle x,y\rangle_\mathcal{E}=\langle x,y\rangle_\mathcal{F}$ for every $x,y\in\mathcal{E}.$
\end{definition}


\begin{definition}\label{d3.6}
An object $(\mathcal{E,A})\in\mathcal{C}^{1}_{H,C^*}$ is an
injective object in $\mathcal{C}_{H,C^*}$ when for every pair of
elements of $\mathcal{C}_{H,C^*}$ such as $(\mathcal{F,B})$ and
$(\mathcal{G,C})$ which $(\mathcal{G,C})$ contained in
$(\mathcal{F,B})$ if there exists a morphism
$(\Phi,\phi):(\mathcal{G,C})\rightarrow(\mathcal{E,A}),$ then
there exists a morphism
$(\Psi,\psi):(\mathcal{F,B})\rightarrow(\mathcal{E,A})$ such that
$\psi$ is an extension of $\phi$ and $\Psi$ is an extension for
$\Phi.$
\end{definition}

We are going to give a characterization of injective objects of
$\mathcal{C}^{1}_{H,C^*}$.  In fact, the next theorem is a
generalization of Theorem \ref{t2}. Before proving the theorem, we
recall some results on injectivity. For an operator space $W,$ its
injective envelope is denoted by $I(W)$ and is the operator space
which contains $W$ such that for every operator space $V$ and
every completely bounded map $\Phi:W\rightarrow V$ there exists a
completely bounded map $\Psi:I(W)\rightarrow V$ such that
$\Psi|_{W}=\Phi.$
 The Paulson operator system associated to $W$ is  $\begin{bmatrix}
\mathbb{C}id & W \\
W^* & \mathbb{C}id
\end{bmatrix}$ and denoted by $S(W).$ First, we prove the following lemma.


\begin{lemma}\label{l3.8}
Let $\mathcal{E}$ be a full right Hilbert $C^*$-module over a
unital $C^*$-algebra $\mathcal A.$
 Then $I(S(\mathcal E))=I(S_{\mathcal A}(\mathcal E)).$
\end{lemma}
\begin{proof}
There exists a Hilbert space $H$ such that $\mathcal{E}$ and
$\mathcal{A}$ be contained in $ \mathcal{B(H)}$ and therefore
$S(\mathcal{E})$ and $S_\mathcal A(\mathcal E)$ can be considered
as subsets of $\mathcal{B}(\mathcal{H}^2).$ Since
$\mathcal{B}(\mathcal H)$ is a unital injective $C^*$-algebra,
there is an injective envelope $I(S(\mathcal E))$ of $S(\mathcal
E)$ such that $S(\mathcal E)\subset I(S(\mathcal
E))\subset\mathcal{B}(\mathcal{H}^2)$ and a completely contractive
idempotent
$\Phi:\mathcal{B}(\mathcal{H}^2)\rightarrow\mathcal{B}(\mathcal{H}^2)
$ which is completely positive and its image is $
I(S(\mathcal{E}))$ and act identically on $S(\mathcal{E}),$ see
\cite[4.2.7]{Blecher}. Since $\Phi$ is a unital completely
positive map and idempotent, there exist unital completely
positive maps
$\varphi_{i}:\mathcal{B(H)}\rightarrow\mathcal{B(H)}$  for $1\leq
i\leq 2,$ and an idempotent
$\varphi:\mathcal{B(H)}\rightarrow\mathcal{B(H)}$ such that
$\Phi=\begin{bmatrix}
  \varphi_{1} & \varphi \\
  \varphi^* & \varphi_{2}
  \end{bmatrix},$ Apply \cite[Corollary 5.2.2]{Effros} or \cite[2.6.16]{Blecher} for $\Phi$ and projections $p=\begin{bmatrix}
  id_\mathcal{H} & 0 \\
  0 & 0
  \end{bmatrix}$ and $id_{\mathcal{H}^2}-p$.
The injective envelope $I(S(\mathcal{E}))$ is a unital
$C^*$-algebra by the product $\circ_\Phi$ defined by
$u_1\circ_\Phi u_2:=\Phi(u_1u_2)$ for every $u_1,u_2\in
I(S(\mathcal{E}))$ and it has the following block-wise structure
$\begin{bmatrix}
  I_{11}(\mathcal{E}) & I(\mathcal{E})\\
  I(\mathcal{E}) & I_{22}(\mathcal{E})
  \end{bmatrix}$
where $I(\mathcal E)=\varphi(\mathcal{B(H)})$ is the injective
envelope of $\mathcal{E}$ (by \cite[4.4.3]{Blecher} see
\cite[4.4.2]{Blecher}) and
$I_{ii}(\mathcal{E})=\varphi_{i}(\mathcal{B(H)})$ for $i=1,2$ are
injective $C^*$-algebras. Since $I(S(\mathcal E))$ is a unital
$C^*$-algebra, $I_{11}(\mathcal E)$ and  $I_{22}(\mathcal E)$ are
unital $C^*$-algebras and $I(\mathcal{E})$ is a Hilbert
$I_{11}(\mathcal E)$-$I_{22}(\mathcal{E})$-bimodule. By the
assumption $\mathcal E$  is full and for every
$u_1,u_2\in\mathcal{E}$ we have
$$\begin{bmatrix}
0 & 0\\
u_2^* & 0
\end{bmatrix}\circ_\Phi\begin{bmatrix}
0 & u_1\\
0 & 0
\end{bmatrix}=\Phi(\begin{bmatrix}
0 & 0\\
u_2^* & 0
\end{bmatrix}\begin{bmatrix}
0 & u_1\\
0 & 0
\end{bmatrix})=\Phi(\begin{bmatrix}
0 & 0\\
0 & \langle u_2,u_1\rangle
\end{bmatrix})=\begin{bmatrix}
0 & 0\\
0 & \varphi_2(\langle u_2,u_1\rangle)
\end{bmatrix}\in\begin{bmatrix}
0 & 0\\
0 & I_{22}(\mathcal E)
\end{bmatrix},$$
thus $\varphi_2(\mathcal A)\subset I_{22}(\mathcal E).$ Note that
$\varphi_2$ is  a unital completely positive map on
$\mathcal{B(H)},$ which is not necessarily multiplicative, but its
restriction on $\mathcal{A}$ is an isometric and multiplicative
map from $\mathcal A$ into $I_{22}(\mathcal E).$ To show this,
note that $\Phi$ is a completely contractive unital idempotent map
which acts identically on $S(\mathcal E)$, thus for every
$u\in\mathcal E$ and $a\in\mathcal A$ by
   \cite[Theorem 4.4.9 (Youngson)]{Blecher} or \cite[Lemma 6.1.2]{Effros} we have
   $$\Phi(\begin{bmatrix}
0 & u\\
0 & 0
\end{bmatrix}\Phi(\begin{bmatrix}
0 & 0\\
0 & a
\end{bmatrix}))=\Phi(\Phi(\begin{bmatrix}
0 & u\\
0 & 0
\end{bmatrix})\begin{bmatrix}
0 & 0\\
0 & a
\end{bmatrix})=\Phi(\begin{bmatrix}
0 & u\\
0 & 0
\end{bmatrix}\begin{bmatrix}
0 & 0\\
0 & a
\end{bmatrix})=\Phi(\begin{bmatrix}
0 & ua\\
0 & 0
\end{bmatrix})=\begin{bmatrix}
0 & ua\\
0 & 0
\end{bmatrix},$$
therefore, if $\varphi_2(a)=0$ for some $a\in\mathcal A,$ then $ua=0$ for every $u\in\mathcal E,$ thus $a=0.$
 Thus $\varphi_2$ is one to one.
  Let $u\in\mathcal E$ and $a,b\in\mathcal A.$
   Put $T_1:=\begin{bmatrix}
0 & u\\
0 & 0
\end{bmatrix},$ $T_2=\begin{bmatrix}
0 & 0\\
0& a
\end{bmatrix}$ and $T_3=\begin{bmatrix}
0 & 0\\
0 & b
\end{bmatrix}.$
Since $T_1T_2T_3$ and $T_1T_2$ belongs to $S(\mathcal E)$ and
$\Phi$ is identity on $S(\mathcal E)$  then \cite[Theorem 4.4.9
(Youngson)]{Blecher} or \cite[Lemma 6.1.2]{Effros} implies that
\begin{equation*}\begin{split}
T_1\circ_\Phi\Phi(T_2T_3)
&=\Phi(T_1\Phi(T_2T_3))=\Phi(\Phi(T_1)T_2T_3)=\Phi(T_1T_2T_3)=T_1T_2T_3
\\&=\Phi((T_1T_2)T_3)=\Phi(\Phi(T_1T_2)T_3)=\Phi(T_1T_2\Phi(T_3))=T_1T_2\circ_\Phi\Phi(T_3)
\\&=\Phi(T_1T_2)\circ_\Phi\Phi(T_3)=\Phi(\Phi(T_1)T_2)\circ_\Phi\Phi(T_3)=
\Phi(T_1\Phi(T_2))\circ_\Phi\Phi(T_3)
\\&=T_1\circ_\Phi\Phi(T_1)\circ_\Phi\Phi(T_3).
\end{split}\end{equation*}

Therefore for every $u\in\mathcal E$ we have
$u\circ_\Phi(\varphi_2(ab)-\varphi_2(a)\circ_\Phi\varphi_2(b))=0$
which implies that
$\varphi_2(ab)=\varphi_2(a)\circ_\Phi\varphi_2(b)$, because
$\mathcal{E}^*\circ_\Phi\mathcal{E}$ is an essential ideal in
$I_{22}(\mathcal E).$ Therefore the restriction of $\varphi_2$ on
$\mathcal{A}$ is an one to one  $*$-homomorphism and therefore it
is an isometry from $\mathcal{A}$ into $I_{22}(\mathcal E)$, thus
$\mathcal {A}\subset I_{22}(\mathcal E).$ Thus $S_{\mathcal
A}(\mathcal E)\subset I(S(\mathcal E))$ which implies that
$I(S_{\mathcal A}(\mathcal E))=I(S(\mathcal E)).$
\end{proof}


\begin{theorem}\label{t3.7}
A given object $(\mathcal{E,A})\in\mathcal{C}^{1}_{H,C^*}$ is
injective if and only if $\mathcal{A}$ and $\mathcal{E}$ are
injective objects in the category of operator spaces.
\end{theorem}

\begin{proof} Let $(\mathcal{E,A})$ be
an injective element in the category $\mathcal{C}_{H,C^*}^1$. By
Lemma \ref{l3.8}, $(\mathcal{E,A})$ is contained in $(I(\mathcal
E),I_{22}(\mathcal E)).$ Thus the identity morphism
$(id,id):(\mathcal{E,A})\rightarrow(\mathcal{E,A})$ has an extension
to a morphism $(\Phi,\phi):(I(\mathcal E),I_{22}(\mathcal
E))\rightarrow(\mathcal{E,A})$. Thus $\phi:I_{22}(\mathcal
E)\rightarrow \mathcal A$ is a completely positive map which extends
the identity map on $\mathcal A$ thus it is unital and
$\Phi:I(\mathcal E)\rightarrow\mathcal E$ is a completely
semi-$\phi$-map  and therefore it is completely contractive. On the
other hand, the inclusion of $\mathcal E$ in $I(\mathcal E)$ is
rigid and $\Phi|_\mathcal{E}=id_\mathcal{E}$, therefore
$\Phi=id_{I(\mathcal E)},$ by \cite[Theorem 6.1.2]{Effros}.
 Thus $\mathcal E=I(\mathcal
E)$, and $\mathcal E$ is injective.

Similarly, using the fact that $(\mathcal {E,A})$ is contained in
$(\mathcal{E},I(\mathcal A))$ we can show that $\mathcal
A=I(\mathcal A)$ and hence $A$ is injective.

Conversely, assume $\mathcal{E}$ and $\mathcal A$ are injective
operator spaces and $\mathcal E$ is a full right Hilbert $\mathcal
A$-module. We show that $(\mathcal{E,A})$ is an injective object
in $\mathcal{C}_{H,C^*}^1.$
 Let $(\mathcal{W, B}),(\mathcal{V, C})\in \mathcal{C}_{H,C^*}^1$ and $(\mathcal{W,B})$ contained in $(\mathcal{V,C})$.
  Assume $(\Phi,\phi)$ is a morphism from $(\mathcal {W,B})$ into $(\mathcal{E,A})$.
Since $\mathcal A$ is an injective $C^*$-algebra, there is a
unital completely positive map
$\psi:\mathcal{C}\rightarrow\mathcal A$ extending $\phi.$
    Note that  $\mathcal B\subset\mathcal C$, thus we can consider $\mathcal W$ as a Hilbert
     $\mathcal C$-module and $\Phi$ is a completely semi-$\psi$-map.
      Thus the map
      $\Lambda:=\begin{bmatrix}
      id & \Phi \\
      \Phi^* & \psi
      \end{bmatrix}:S_\mathcal{C}(\mathcal W)\rightarrow S_\mathcal{A}(\mathcal E)$
    is a unital completely positive map (by Proposition \ref{p3.1}).
     On the other hand
      $S_\mathcal{A}(\mathcal E)\subset I(S_\mathcal{A}(\mathcal E))=I(S(\mathcal E))=\begin{bmatrix}
      I_{11}(\mathcal E) & I(\mathcal E)\\
      I(\mathcal E) & I_{22}(\mathcal E)\end{bmatrix}.$
Note that $S_\mathcal{C}(\mathcal W)\subset S_\mathcal{C}(\mathcal
V)$ and $I(S(\mathcal E))$ is injective, thus there is a unital
completely positive map $\Theta:S_\mathcal{C}(\mathcal
V)\rightarrow I(S(\mathcal E))$ which extends $\Lambda.$
 It is obvious that $\Theta$ has the matrix decomposition form $\begin{bmatrix}
      id & \Psi \\
      \Psi^* & \psi
      \end{bmatrix}$
for some linear map $\Psi:\mathcal{V}\rightarrow I(\mathcal E),$
but $\mathcal E$ is injective, thus $I(\mathcal E)=\mathcal E$
and, there exists a linear map
$\Psi:\mathcal{V}\rightarrow\mathcal E$ such that $\begin{bmatrix}
      id & \Psi \\
      \Psi^* & \psi
      \end{bmatrix}:S_\mathcal{C}(\mathcal V)\to S_\mathcal{A}(\mathcal E)$
is a unital completely positive map. Now Proposition \ref{p3.1}
implies that $\Psi$ is a completely semi-$\psi$-map, extending
$\Phi.$ Therefore $(\mathcal{E,A})$ is an injective object in
$\mathcal{C}_{H,C^*}^1.$
\end{proof}
Note that $\mathcal{C}_{H,C^*}^1$ is different from the category
of operator spaces, we show this by an example of an injective
object in $\mathcal{C}_{H,C^*}^1$ which its corresponding object
is not a injective operator space.
\begin{ex}\label{e3.9}
Assume $\mathcal H$ be a infinite dimensional Hilbert space. Put
$\mathcal E=\mathcal{B}(\mathbb{C},\mathcal H).$ By Theorem 2.4 or
theorem 3.7 $(\mathcal{B}(\mathbb{C},\mathcal H),\mathbb{C})$ is an
injective object in $\mathcal{C}_{H,C^*}$ and
$\mathcal{C}_{H,C^*}^1$. But $S_{\mathbb{C}}(\mathcal E)=S(\mathcal
E)$ and $I(S(\mathcal E))=\begin{bmatrix}
 \mathcal{B(H)} & \mathcal{E}\\
 \mathcal{E}^* & \mathbb{C}
 \end{bmatrix}.$ Thus $S(\mathcal E)$ is not an injective operator system.
 \end{ex}

\subsection*{Acknowledgment}
The research of the first author was in part supported by a grant
from IPM (No. 94470046).


\end{document}